\documentclass[a4paper,12pt,reqno]{amsart}
\usepackage{amsfonts}
\usepackage{amsmath}
\usepackage{amssymb}
\usepackage[a4paper]{geometry}
\usepackage{mathrsfs}
\usepackage{hyperref}
\renewcommand\eqref[1]{(\ref{#1})} 
\numberwithin{equation}{section}
\theoremstyle{plain}
\newtheorem{thm}{Theorem}[section]
\newtheorem{prop}[thm]{Proposition}

\theoremstyle{definition}

\renewcommand{\wp}{\mathfrak S}
\newcommand{\Rn}{\mathbb R^{n}}

\begin{document}

   \title[$L^{p}$-Caffarelli-Kohn-Nirenberg type inequalities]
   {$L^{p}$-Caffarelli-Kohn-Nirenberg type inequalities
   on homogeneous groups}
\author[T. Ozawa]{Tohru Ozawa}
\address{
  Tohru Ozawa:
  \endgraf
  Department of Applied Physics
  \endgraf
  Waseda University
  \endgraf
  Tokyo 169-8555
  \endgraf
  Japan
  \endgraf
  {\it E-mail address} {\rm txozawa@waseda.jp}
  }
\author[M. Ruzhansky]{Michael Ruzhansky}
\address{
  Michael Ruzhansky:
  \endgraf
  Department of Mathematics
  \endgraf
  Imperial College London
  \endgraf
  180 Queen's Gate, London SW7 2AZ
  \endgraf
  United Kingdom
  \endgraf
  {\it E-mail address} {\rm m.ruzhansky@imperial.ac.uk}
  }
\author[D. Suragan]{Durvudkhan Suragan}
\address{
  Durvudkhan Suragan:
  \endgraf
  Institute of Mathematics and Mathematical Modelling
  \endgraf
  125 Pushkin str.
  \endgraf
  050010 Almaty
  \endgraf
  Kazakhstan
  \endgraf
  and
  \endgraf
  Department of Mathematics
  \endgraf
  Imperial College London
  \endgraf
  180 Queen's Gate, London SW7 2AZ
  \endgraf
  United Kingdom
  \endgraf
  {\it E-mail address} {\rm d.suragan@imperial.ac.uk}
  }

\thanks{The second and third authors were supported in parts by the EPSRC
 grant EP/K039407/1 and by the Leverhulme Grant RPG-2014-02,
 as well as by the MESRK grant 5127/GF4. No new data was collected or generated during the course of research.}


     \begin{abstract}
     We prove $L^{p}$-Caffarelli-Kohn-Nirenberg type inequalities on homogeneous groups, which is one of most general subclasses of nilpotent Lie groups, all with sharp constants.  We also discuss some of their consequences.
     Already in the Abelian cases of isotropic or anisotropic
$\mathbb{R}^{n}$ our results provide new conclusions in view
of the arbitrariness of the choice of the not necessarily
Euclidean quasi-norm.
     \end{abstract}
     \maketitle

\section{Introduction}

Consider the following weighted Hardy-Sobolev type
inequalities due to Caffarelli, Kohn and Nirenberg \cite{CKN-1984}:
For all $f\in C_{0}^{\infty}(\mathbb{R}^{n})$ it holds:
\begin{equation}
\left(\int_{\mathbb{R}^{n}}\|x\|^{-p\beta}|f|^{ p} dx\right)^{\frac{2}{p}}\leq C_{\alpha,\beta}\int_{\mathbb{R}^{n}}\|x\|^{-2\alpha}|\nabla f|^{2}dx,
\end{equation}
where for $n\geq3$:
$$-\infty<\alpha<\frac{n-2}{2},\; \alpha\leq \beta\leq\alpha+1,\;{\rm and} \; p=\frac{2n}{n-2+2(\beta-\alpha)},$$
and for
$n=2$:
$$-\infty<\alpha<0,\; \alpha<\beta\leq\alpha+1,\;{\rm and}\; p=\frac{2}{\beta-\alpha},$$
and where $\|x\|=\sqrt{x^{2}_{1}+\ldots+x^{2}_{n}}$.
Nowadays there is a lot of literature on Caffarelli-Kohn-Nirenberg type inequalities and their applications.
In the case $p=2$ (see e. g. \cite{WW-2003}), for all $f\in C_{0}^{\infty}(\mathbb{R}^{n})$, one has

\begin{equation}
\int_{\mathbb{R}^{n}}\|x\|^{-2(\alpha+1)}|f|^{2}dx\leq \widetilde{C}_{\alpha}\int_{\mathbb{R}^{n}}\|x\|^{-2\alpha}|\nabla f|^{2}dx,
\end{equation}
with any $n\geq2$ and $-\infty<\alpha<0$, which in turn can be presented for any $f\in C_{0}^{\infty}(\mathbb{R}^{n}\backslash\{0\})$ as
\begin{equation}\label{CKNinRn}
\left\|\frac{1}{\|x\|^{\alpha+1}}|f|\right\|_{L^{2}(\mathbb{R}^{n})}\leq C_{\alpha}\left\|\frac{1}{\|x\|^{\alpha}}|\nabla f|\right\|_{L^{2}(\mathbb{R}^{n})},
\end{equation}
all $\alpha\in\mathbb{R}^{n}$.

Motivating the development of the analysis associated to homogeneous groups in \cite{FS-Hardy}, Folland and Stein raised an important question of determining which elements of the classical harmonic analysis do depend only on the group and the dilation structures. The natural setting for this kind of problems is that of homogeneous groups, in particular, including the cases of anisotropic structures on $\mathbb R^{n}$. In this paper we show that the Caffarelli-Kohn-Nirenberg inequality continues to hold in the setting of homogeneous groups. In particular, it has to also hold on anisotropic $\mathbb R^{n}$, with a number of different consequences.

Recently a homogeneous group version of the inequality \eqref{CKNinRn} was obtained in the work \cite{Ruzhansky-Suragan:identities}, i.e., it was proved that if $\mathbb{G}$ is a homogeneous group
of homogeneous dimension $Q$,
then for all  $f\in C^{\infty}_{0}(\mathbb{G}\backslash\{0\})$ and for every homogeneous quasi-norm $|\cdot|$ on $\mathbb{G}$ we have
\begin{equation}\label{awHardyeq-g}
\frac{|Q-2-2\alpha|}{2}\left\|\frac{f}{|x|^{\alpha+1}}\right\|_{L^{2}(\mathbb{G})}\leq
\left\|\frac{1}{|x|^{\alpha}}\mathcal{R} f\right\|_{L^{2}(\mathbb{G})},\quad \forall\alpha\in\mathbb R,
\end{equation}
where $\mathcal{R}$ is defined by \eqref{EQ:Euler}.
Note that if $\alpha\neq\frac{Q-2}{2}$, then the constant in \eqref{awHardyeq-g} is sharp for any homogeneous quasi-norm $|\cdot|$ on $\mathbb{G}$.
In the abelian case ${\mathbb G}=(\mathbb R^{n},+)$, we have
$Q=n$, $e(x)=x=(x_{1},\ldots,x_{n})$, so for each $\alpha\in\mathbb{R}$ with $\alpha\neq\frac{n-2}{2}$ and for any homogeneous quasi-norm $|\cdot|$ on $\mathbb R^{n}$ the inequality \eqref{awHardyeq-g} implies
a new inequality with the optimal constant:
\begin{equation}\label{Hardy-r}
\frac{|n-2-2\alpha|}{2}\left\|\frac{f}{|x|^{\alpha+1}}\right\|_{L^{2}(\mathbb{R}^{n})}\leq
\left\|\frac{1}{|x|^{\alpha}}\frac{x}{|x|}\cdot\nabla f\right\|_{L^{2}(\mathbb{R}^{n})}.
\end{equation}
In turn, by using Schwarz's inequality with the standard Euclidean distance $\|x\|=\sqrt{x^{2}_{1}+\ldots+x^{2}_{n}}$, this implies the $L^{2}$ Caffarelli-Kohn-Nirenberg inequality \cite{CKN-1984} for $\mathbb{G}\equiv\mathbb{R}^{n}$ with the optimal constant:
\begin{equation}\label{CKN}
\frac{|n-2-2\alpha|}{2}\left\|\frac{f}{\|x\|^{\alpha+1}}\right\|_{L^{2}(\mathbb{R}^{n})}\leq
\left\|\frac{1}{\|x\|^{\alpha}}\nabla f\right\|_{L^{2}(\mathbb{R}^{n})},\;\forall\alpha\in\mathbb{R},
\end{equation}
for all $f\in C_{0}^{\infty}(\mathbb{R}^{n}\backslash\{0\}).$
Here optimality $\frac{|n-2-2\alpha|}{2}$ of the constant  was proved in \cite[Theorem 1.1. (ii)]{CW-2001}.
In addition, we can also note that the analysis in these type inequalities and improvements of their remainder terms has a long history, initiated by Brezis and Nirenberg in \cite{Brez3} and then in \cite{Brez1} for Sobolev inequalities, and in \cite{Brez4} for Hardy inequalities, see also \cite{Brez2}, with many subsequent works in this subject.

The $L^{2}$-inequality \eqref{CKNinRn} in the anisotropic setting as well as in the more abstract setting of homogeneous groups was analysed in \cite{Ruzhansky-Suragan:L2-CKN} by using an explicit formula for the remainder that is available in the case of $L^{2}$-spaces. Such a remainder formula fails in the scale of $L^{p}$-spaces for $p\not=2$ and, therefore, in this paper we approach these inequality by a different method.

Thus, the main aim of this paper is to extend the above inequality \eqref{awHardyeq-g} to the general $L^{p}$ case for all $1<p<\infty$ by a different approach.
Here, the first result of this paper is:
For each $f\in C^{\infty}_{0}(\mathbb{G}\backslash\{0\}),$ $1<p<\infty$,
and any homogeneous quasi-norm $|\cdot|$ on $\mathbb{G}$ we have
\begin{equation}\label{eqLpCKN}
\frac{|Q-\gamma|}{p}
\left\|\frac{f}{|x|^{\frac{\gamma}{p}}}\right\|^{p}_{L^{p}(\mathbb{G})}\leq
\left\|\frac{1}{|x|^{\alpha}}\mathcal{R} f\right\|_{L^{p}(\mathbb{G})}\left\|\frac{f}{|x|^{\frac{\beta}{p-1}}}\right\|^{p-1}_{L^{p}(\mathbb{G})},\quad \forall\alpha,\,\beta\in \mathbb{R},
\end{equation}
where $\gamma=\alpha+\beta+1$, and the constant
$\frac{|Q-\gamma|}{p}$ is sharp if $\gamma\neq Q$. Here $\mathcal{R}$ is the radial derivative operator on $\mathbb{G}$ defined by \eqref{EQ:Euler}.

All above inequalities are generalisations of the classical Hardy inequality, which takes the form
\begin{equation}\label{HRn-p}
\left\|\frac{f(x)}{\|x\|}\right\|_{L^{p}(\Rn)} \leq \frac{p}{n-p}\left\| \nabla f\right\|_{L^{p}(\Rn)},\quad
n\geq 2,\; 1\leq p<n,
\end{equation}
where $\nabla$ is the standard gradient in
$\mathbb{R}^{n}$, $f\in C_{0}^{\infty}(\mathbb{R}^{n}
\backslash \{0\})$, $\|x\|=\sqrt{x_{1}^{2}+...+x_{n}^{2}},$
and the constant $\frac{p}{n-p}$ is known to be sharp.

We refer to a recent interesting paper of Hoffmann-Ostenhof and Laptev \cite{Laptev15} on this subject for inequalities with weights, to \cite{HHLT-Hardy-many-particles} for many-particle versions, to Ekholm, Kova{\v{r}}{\'{\i}}k and Laptev \cite{EKL:Hardy-p-Lap} for $p$-Laplacian interpretations, and to many further references
therein and otherwise.
We refer also to more recent preprints \cite{Ruzhansky-Suragan:Layers}-\cite{Ruzhansky-Suragan:uncertainty} and references therein for the story behind Hardy type inequalities on nilpotent Lie groups.

Before giving preliminaries for stating our results let us mention another observation that the Hardy inequality \eqref{HRn-p} can be sharpened to the inequality
\begin{equation}\label{HRn-p-sh}
\left\|\frac{f(x)}{\|x\|}\right\|_{L^{p}(\Rn)} \leq \frac{p}{n-p}\left\| \frac{x}{\|x\|}\cdot\nabla f\right\|_{L^{p}(\Rn)},\quad
n\geq 2,\; 1\leq p<n.
\end{equation}
It is clear that \eqref{HRn-p-sh} implies \eqref{HRn-p} since the function $\frac{x}{\|x\|}$ is bounded.
The remainder terms for \eqref{HRn-p-sh} have been analysed by Ioku, Ishiwata and Ozawa
\cite{IIO:Lp-Hardy}, see also Machihara, Ozawa and Wadade \cite{MOW:Hardy-Hayashi}.

 One of the results in \cite{Ruzhansky-Suragan:identities} was that if $\mathbb G$ is any homogeneous group and $|\cdot|$ is a homogeneous quasi-norm on $\mathbb G$, as an analogue of \eqref{HRn-p-sh}
we obtain the following generalised $L^{p}$-Hardy
inequality:
\begin{equation}\label{iLpHardyeq0}
\left\|\frac{f}{|x|}\right\|_{L^{p}(\mathbb{G})}\leq\frac{p}{Q-p}\left\|\frac{e(x)}{|x|}\cdot A\nabla f\right\|_{L^{p}(\mathbb{G})},
\quad 1<p<Q,
\end{equation}
for all $f\in C_{0}^{\infty}(\mathbb{G}\backslash\{0\}).$
Here $\nabla=(X_{1},\ldots,X_{n})$ is a gradient on $\mathbb G$ with
a basis $\{X_{1},\ldots,X_{n}\}$ of the Lie algebra
 $\mathfrak{g}$ of $\mathbb{G}$,
$A$ is the $n$-diagonal matrix
\begin{equation}\label{EQ:mA0}
A={\rm diag}(\nu_{1},\ldots,\nu_{n}),
\end{equation}
where $\nu_{k}$ is the homogeneous degree of $X_{k}$, and
$$Q = {\rm Tr}\,A=\nu_{1}+\cdots+\nu_{n}$$
is the homogeneous dimension of $\mathbb G$.
We note that the exponential mapping ${\exp}_{\mathbb{G}}:\mathfrak g\to\mathbb G$ is a global diffeomorphism and the vector $e(x)=(e_{1}(x),\ldots,e_{n}(x))$
is the decomposition of its inverse ${\exp}_{\mathbb{G}}^{-1}$ with respect to the basis
$\{X_{1},\ldots,X_{n}\}$, namely, $e(x)$ is determined by
$${\exp}_{\mathbb{G}}^{-1}(x)=e(x)\cdot \nabla\equiv\sum_{j=1}^{n}e_{j}(x)X_{j}.$$

For $p=n$ or $p=Q$ the inequalities \eqref{HRn-p} and \eqref{iLpHardyeq0}
 fail for any constant. The critical versions of
\eqref{HRn-p-sh} with $p=n$ were investigated by Ioku, Ishiwata and Ozawa \cite{IIO}.
Their generalisations as well as a number of other critical (logarithmic) Hardy inequalities on homogeneous groups were obtained in recent works \cite{Ruzhansky-Suragan:critical}, \cite{Ruzhansky-Suragan:identities} and \cite{Ruzhansky-Suragan:uncertainty}.
Here we only mention a related family of logarithmic
Hardy inequalities
\begin{equation}\label{LH2p}
\qquad  \underset{R>0}{\sup}\left\|\frac{f-f_{R}}{|x|^{\frac{Q}{p}}{\log}\frac{R}{|x|}}
\right\|_{L^{p}(\mathbb{G})}\leq
\frac{p}{p-1}\left\| \frac{1}{|x|^{\frac{Q}{p}-1}}\left(\frac{e(x)}{|x|}\cdot A\nabla\right) f\right\|_{L^{p}(\mathbb{G})},
\end{equation}
for all $1< p<\infty$, where $f_{R}=f(R\frac{x}{|x|})$. We refer to
\cite{Ruzhansky-Suragan:critical} for further explanations and extensions but only mention here that for $p=Q$ the inequality \eqref{LH2p} gives a critical case of Hardy's inequalities \eqref{iLpHardyeq0}.

In Section \ref{Sec2} we give some necessary tools on homogeneous
groups and fix the notation. In Section \ref{Sec3} we present $L^{p}$-Caffarelli-Kohn-Nirenberg type inequalities on the homogeneous group $\mathbb{G}$ and then discuss their consequences and proofs. In Section \ref{Sec4} we discuss higher order cases.

\section{Preliminaries}
\label{Sec2}

In this standard preliminary section we very shortly recall some basics details of homogeneous groups.
The general analysis on homogeneous groups was developed by Folland and Stein in their book
\cite{FS-Hardy}, and we also refer for more recent developments to the monograph \cite{FR} by V\'eronique Fischer and the second named author.

It is known that a dilation family of a Lie algebra $\mathfrak{g}$
has the following representation
$$D_{\lambda}={\rm Exp}(A \,{\rm ln}\lambda)=\sum_{k=0}^{\infty}
\frac{1}{k!}({\rm ln}(\lambda) A)^{k},$$
where $A$ is a diagonalisable positive linear operator on $\mathfrak{g}$,
and each family of linear mappings $D_{\lambda}$ is a morphism of $\mathfrak{g}$,
i.e. a linear mapping
from $\mathfrak{g}$ to itself with the property
$$\forall X,Y\in \mathfrak{g},\, \lambda>0,\;
[D_{\lambda}X, D_{\lambda}Y]=D_{\lambda}[X,Y],$$
where $[X,Y]:=XY-YX$ is the Lie bracket.
Shortly, a {\em homogeneous group} is a connected simply connected Lie group whose
Lie algebra is equipped with dilations.
It induces the dilation structure on $\mathbb G$ which we continue to denote by $D_{\lambda}(x)$ or simply by $\lambda x$. We denote by
$$Q := {\rm Tr}\,A$$
the homogeneous dimension of $\mathbb G$.
We also recall that the standard Lebesque measure on $\mathbb R^{N}$ is the Haar measure for $\mathbb{G}$ (see, e.g. \cite[Proposition 1.6.6]{FR}).

Let us fix a basis $\{X_{1},\ldots,X_{n}\}$ of the Lie algebra $\mathfrak{g}$ of the homogeneous group $\mathbb{G}$
such that
$$AX_{k}=\nu_{k}X_{k}$$
for each $1\leq k\leq n$, so that $A$ can be taken to be
\begin{equation}
A={\rm diag}(\nu_{1},\ldots,\nu_{n}).
\end{equation}
Then each $X_{k}$ is homogeneous of degree $\nu_{k}$ and also
\begin{equation}
Q=\nu_{1}+\cdots+\nu_{n},
\end{equation}
which is called a homogeneous dimension of $\mathbb{G}$.
Since homogeneous groups are nilpotent   
the exponential mapping $\exp_{\mathbb G}:\mathfrak g\to\mathbb G$ is a global diffeomorphism.
In addition, the decomposition of ${\exp}_{\mathbb{G}}^{-1}(x)$ in the Lie algebra $\mathfrak g$ defines the vector
$$e(x)=(e_{1}(x),\ldots,e_{n}(x))$$
by the formula
$${\exp}_{\mathbb{G}}^{-1}(x)=e(x)\cdot \nabla\equiv\sum_{j=1}^{n}e_{j}(x)X_{j},$$
where $\nabla=(X_{1},\ldots,X_{n})$.
Alternatively, this means the equality
$$x={\exp}_{\mathbb{G}}\left(e_{1}(x)X_{1}+\ldots+e_{n}(x)X_{n}\right).$$
Using homogeneity we have
$$rx:=D_{r}(x)={\exp}_{\mathbb{G}}\left(r^{\nu_{1}}e_{1}(x)X_{1}+\ldots
+r^{\nu_{n}}e_{n}(x)X_{n}\right),$$
that is,
$$
e(rx)=(r^{\nu_{1}}e_{1}(x),\ldots,r^{\nu_{n}}e_{n}(x)).
$$
Since $r>0$ is arbitrary, without loss of generality taking $|x|=1$, a direct calculation shows that
\begin{align*}
\frac{d}{dr}f(rx) & =  \frac{d}{dr}f({\exp}_{\mathbb{G}}
\left(r^{\nu_{1}}e_{1}(x)X_{1}+\ldots
+r^{\nu_{n}}e_{n}(x)X_{n}\right)) \\
& =  \left[(\nu_{1}r^{\nu_{1}-1}e_{1}(x)X_{1}+\ldots
+\nu_{n}r^{\nu_{n}-1}e_{n}(x)X_{n})f\right](rx) \\
& =
\frac{1}{r}\left(e(rx)\cdot A\nabla
\right)f(rx).
\end{align*}
Using the notation
\begin{equation}\label{EQ:Euler}
\mathcal{R} :=\frac{e(x)}{|x|}\cdot A\nabla=\sum_{j=1}^{n} \nu_{j} \frac{e_{j}(x)}{|x|}X_{j},
\end{equation}
we obtain that
\begin{equation}\label{dfdr}
\frac{d}{dr}f(rx)=\frac{1}{r}\left(e(rx)\cdot A\nabla
\right)f(rx)=\mathcal{R} f(rx).
\end{equation}
That is, the operator $\mathcal{R}$ plays the role of the radial derivative on $\mathbb G$.
It follows from \eqref{EQ:Euler} that $\mathcal{R} $ is homogeneous of order $-1$.
It is known that every homogeneous group $\mathbb{G}$ admits a homogeneous
quasi-norm $|\cdot|$.
The $|\cdot|$-ball centred at $x\in\mathbb{G}$ with radius $R > 0$
can be defined by
$$B(x,R):=\{y\in \mathbb{G}: |x^{-1}y|<R\}.$$
The following polar decomposition was established in \cite{FS-Hardy}
(see also \cite[Section 3.1.7]{FR}).

\begin{prop}\label{polarinteg}
Let $\mathbb{G}$
be a homogeneous group equipped with a homogeneous
quasi-norm $\mid\cdot\mid$. Then there is a (unique)
positive Borel measure $\sigma$ on the
unit pseudo sphere
\begin{equation}\label{EQ:sphere}
\wp:=\{x\in \mathbb{G}:\,|x|=1\},
\end{equation}
such that for all $f\in L^{1}(\mathbb{G})$, we have
\begin{equation}
\int_{\mathbb{G}}f(x)dx=\int_{0}^{\infty}
\int_{\wp}f(ry)r^{Q-1}d\sigma(y)dr.
\end{equation}
\end{prop}

\section{$L^{p}$-Caffarelli-Kohn-Nirenberg type inequalities and consequences}
\label{Sec3}
In this section and in the sequel we adopt all the notation introduced in Section \ref{Sec2} concerning homogeneous groups and the operator $\mathcal{R}$. We formulate the following $L^{p}$-Caffarelli-Kohn-Nirenberg type inequalities on the homogeneous group $\mathbb{G}$ and then discuss their consequences and proofs.

\begin{thm}\label{LpCKN}
Let $\mathbb{G}$ be a homogeneous group
of homogeneous dimension $Q$ and let $\alpha,\,\beta\in \mathbb{R}$.
Then for all complex-valued functions $f\in C^{\infty}_{0}(\mathbb{G}\backslash\{0\}),$ $1<p<\infty,$
and any homogeneous quasi-norm $|\cdot|$ on $\mathbb{G}$ we have
\begin{equation}\label{eqLpCKN}
\frac{|Q-\gamma|}{p}
\left\|\frac{f}{|x|^{\frac{\gamma}{p}}}\right\|^{p}_{L^{p}(\mathbb{G})}\leq
\left\|\frac{1}{|x|^{\alpha}}\mathcal{R} f\right\|_{L^{p}(\mathbb{G})}\left\|\frac{f}{|x|^{\frac{\beta}{p-1}}}\right\|^{p-1}_{L^{p}(\mathbb{G})},
\end{equation}
where $\gamma=\alpha+\beta+1$. If $\gamma\neq Q$ then the constant $\frac{|Q-\gamma|}{p}$ is sharp.
\end{thm}

In the abelian case ${\mathbb G}=(\mathbb R^{n},+)$, we have
$Q=n$, $e(x)=x=(x_{1},\ldots,x_{n})$, so for any homogeneous quasi-norm $|\cdot|$ on $\mathbb R^{n}$ \eqref{eqLpCKN} implies
a new inequality with the optimal constant:
\begin{equation}\label{Hardy-r}
\frac{|n-\gamma|}{p}
\left\|\frac{f}{|x|^{\frac{\gamma}{p}}}\right\|^{p}_{L^{p}(\mathbb{R}^{n})}\leq
\left\|\frac{1}{|x|^{\alpha}}\frac{x}{|x|}\cdot\nabla f\right\|_{L^{p}(\mathbb{R}^{n})}
\left\|\frac{f}{|x|^{\frac{\beta}{p-1}}}
\right\|^{p-1}_{L^{p}(\mathbb{R}^{n})},
\end{equation}
which in turn, by using Schwarz's inequality with the standard Euclidean distance $\|x\|=\sqrt{x^{2}_{1}+\ldots+x^{2}_{n}}$, implies the $L^{p}$-Caffarelli-Kohn-Nirenberg type inequality (see \cite{DJSJ-2013} and \cite{Costa-2008}) for $\mathbb{G}\equiv\mathbb{R}^{n}$ with the sharp constant:
\begin{equation}\label{CKN}
\frac{|n-\gamma|}{p}
\left\|\frac{f}{\|x\|^{\frac{\gamma}{p}}}\right\|^{p}_{L^{p}(\mathbb{R}^{n})}\leq
\left\|\frac{1}{\|x\|^{\alpha}}\nabla f\right\|_{L^{p}(\mathbb{R}^{n})}\left\|\frac{f}{\|x\|^{\frac{\beta}{p-1}}}\right\|^{p-1}_{L^{p}(\mathbb{R}^{n})},
\end{equation}
for all $f\in C_{0}^{\infty}(\mathbb{R}^{n}\backslash\{0\}).$

When $\alpha=0,\, \beta=p-1$ and $1<p<Q$, the inequality \eqref{eqLpCKN} gives the homogeneous group version of $L^{p}$-Hardy inequality
\begin{equation}\label{47-1}
\left\|\frac{1}{|x|}f
\right\|_{L^{p}(\mathbb{G})}\leq\frac{p}{Q-p}\left\|\mathcal{R} f\right\|_{L^{p}(\mathbb{G})},
\end{equation}
again with $\frac{p}{Q-p}$ being the best constant
(see \cite{Ruzhansky-Suragan:critical}-\cite{Ruzhansky-Suragan:horizontal} for weighted, critical, higher order cases, horizontal cases and their applications in different settings).
Note that in comparison to stratified (Carnot) group versions, here the constant is best
for any homogeneous quasi-norm $|\cdot|$.

In the abelian case ${\mathbb G}=(\mathbb R^{n},+)$, $n\geq 3$, we have
$Q=n$, $e(x)=x=(x_{1},\ldots,x_{n})$, so for any quasi-norm $|\cdot|$ on $\mathbb R^{n}$ \eqref{47-1} implies
the new inequality:
\begin{equation}\label{Hardy-p}
\left\|\frac{f}{|x|}\right\|_{L^{p}(\mathbb{R}^{n})}\leq \frac{p}{n-p}
\left\|\frac{x}{|x|}\cdot\nabla f\right\|_{L^{p}
(\mathbb{R}^{n})}.
\end{equation}
In turn, by using Schwarz's inequality with the standard Euclidean distance $\|x\|=\sqrt{x^{2}_{1}+\ldots+x^{2}_{n}}$, it implies the classical Hardy inequality for $\mathbb{G}\equiv\mathbb{R}^{n}$:
\begin{equation*}\label{Hardy}
\left\|\frac{f}{\|x\|}\right\|_{L^{p}(\mathbb{R}^{n})}\leq
\frac{p}{n-p}\left\|\nabla f\right\|_{L^{p}(\mathbb{R}^{n})},
\end{equation*}
for all $f\in C_{0}^{\infty}(\mathbb{R}^{n}\backslash\{0\}).$
When $|x|\equiv\|x\|$ the remainder terms for \eqref{Hardy-p}, that is, the exact formulae of the difference between the right hand side and the left hand side of the inequality, have been analysed by Ioku, Ishiwata and Ozawa
\cite{IIO:Lp-Hardy}, see also Machihara, Ozawa and Wadade \cite{MOW:Hardy-Hayashi} as well as \cite{IIO}.

The inequality \eqref{47-1} also implies the following Heisenberg-Pauli-Weyl type  uncertainly principle on homogeneous groups (see e.g. \cite{Ciatti-Cowling-Ricci}, \cite{Ruzhansky-Suragan:Hardy} and \cite{Ruzhansky-Suragan:Layers} for versions of abelian and stratified groups):
For each $f\in C^{\infty}_{0}(\mathbb{G}\backslash\{0\})$ and any homogeneous quasi-norm $|\cdot|$ on $\mathbb{G}$, using H\"older's inequality and \eqref{47-1}, we have
\begin{multline}
\left\|f\right\|^{2}_{L^{2}(\mathbb{G})}
\leq \left\|\frac{1}{|x|} f\right\|_{L^{p}(\mathbb{G})}\left\||x| f\right\|_{L^{\frac{p}{p-1}}(\mathbb{G})}
\\\leq\frac{p}{Q-p}\left\|\mathcal{R} f\right\|_{L^{p}(\mathbb{G})}\left\||x| f\right\|_{L^{\frac{p}{p-1}}(\mathbb{G})},\quad 1<p<Q,
\end{multline}
that is,
\begin{equation}\label{UP1p}
\left\|f\right\|^{2}_{L^{2}(\mathbb{G})}
\leq\frac{p}{Q-p}\left\|\mathcal{R} f\right\|_{L^{p}(\mathbb{G})}\left\||x| f\right\|_{L^{\frac{p}{p-1}}(\mathbb{G})},\quad 1<p<Q.
\end{equation}
In the abelian case ${\mathbb G}=(\mathbb R^{n},+)$, taking
$Q=n$ and $e(x)=x$, we obtain that \eqref{UP1p} with $p=2$ implies
the uncertainly principle with any quasi-norm $|x|$:
\begin{equation}\label{UPRn-2}
\left(\int_{\mathbb R^{n}}
 |u(x)|^{2} dx\right)^{2}\leq\left(\frac{2}{n-2}\right)^{2}\int_{\mathbb R^{n}}\left|\frac{x}{|x|}\cdot\nabla u(x)\right|^{2}dx
\int_{\mathbb R^{n}} |x|^{2} |u(x)|^{2}dx.
\end{equation}
In turn it implies
the classical
uncertainty principle for $\mathbb{G}\equiv\mathbb R^{n}$ with the standard Euclidean distance $\|x\|$:
\begin{equation*}\label{UPRn}
\left(\int_{\mathbb R^{n}}
 |u(x)|^{2} dx\right)^{2}\leq\left(\frac{2}{n-2}\right)^{2}\int_{\mathbb R^{n}}|\nabla u(x)|^{2}dx
\int_{\mathbb R^{n}} \|x\|^{2} |u(x)|^{2}dx,
\end{equation*}
which is the Heisenberg-Pauli-Weyl uncertainly principle on $\mathbb R^{n}$.

On the other hand, directly from the inequality \eqref{eqLpCKN} we can obtain
a number of Heisenberg-Pauli-Weyl type uncertainly inequities which have
various consequences and aplications.
For example, when $\alpha p=\alpha+\beta+1$, we have
\begin{equation}\label{HPW1}
\frac{|Q-\alpha p|}{p}\left\|\frac{f}{|x|^{\alpha}}\right\|^{p}_{L^{p}(\mathbb{G})}
\leq\left\|\frac{\mathcal{R} f}{|x|^{\alpha}}\right\|_{L^{p}(\mathbb{G})}\left\||x|^{\frac{1}{p-1}}
\frac{f}{|x|^{\alpha}}
\right\|^{p-1}_{L^{p}(\mathbb{G})},
\end{equation}
and, on the other hand, if $0=\alpha+\beta+1$ and $\alpha=-p$ then
\begin{equation}\label{HPW2}
\frac{Q}{p}\left\|f\right\|^{p}_{L^{p}(\mathbb{G})}
\leq\left\||x|^{p}\mathcal{R} f\right\|_{L^{p}(\mathbb{G})}\left\| \frac{f}{|x|}\right\|^{p-1}_{L^{p}(\mathbb{G})},
\end{equation}
all with sharp constants.

\begin{proof}[Proof of Theorem \ref{LpCKN}]
We may assume that $\gamma\neq Q$ since for $\gamma=Q$ there is nothing to prove. Introducing polar coordinates $(r,y)=(|x|, \frac{x}{\mid x\mid})\in (0,\infty)\times\wp$ on $\mathbb{G}$, where $\wp$ is the pseudo-sphere in \eqref{EQ:sphere}, and using Proposition \ref{polarinteg}
one calculates
$$
\int_{\mathbb{G}}
\frac{|f(x)|^{p}}
{|x|^{\gamma}}dx
=\int_{0}^{\infty}\int_{\wp}
\frac{|f(ry)|^{p}}
{r^{\gamma}}r^{Q-1}d\sigma(y)dr
$$
$$
=\frac{1}{Q-\gamma}\int_{0}^{\infty}\int_{\wp}
|f(ry)|^{p} \frac{d\,r^{Q-\gamma}}{dr}d\sigma(y)dr
$$$$=-\frac{1}{Q-\gamma}{\rm Re}\int_{0}^{\infty}\int_{\wp} pf(ry)|f(ry)|^{p-2}\left(\overline{\frac{df(ry)}{dr}}\right)\frac{1}{r^{\gamma-1}}
r^{Q-1}d\sigma(y)dr
$$
$$
=-\frac{p}{Q-\gamma} {\rm Re}\int_{\mathbb{G}}
f(x)\frac{|f(x)|^{p-2}}{|x|^{\gamma-1}}
\left(\overline{\frac{1}{|x|}\left(e(x)\cdot A\nabla
\right)f(x)}\right)dx
$$
$$
\leq\left|\frac{p}{Q-\gamma}\right| \int_{\mathbb{G}}
\frac{|f(x)|^{p-1}}{|x|^{\gamma-1}}
\left|\frac{1}{|x|}\left(e(x)\cdot A\nabla
\right)f(x)\right|dx
$$
$$
=\left|\frac{p}{Q-\gamma}\right| \int_{\mathbb{G}}
\frac{|f(x)|^{p-1}}{|x|^{\alpha+\beta}}
\left|\mathcal{R}f(x)\right|dx
$$
$$
\leq \left|\frac{p}{Q-\gamma}\right| \left(\int_{\mathbb{G}}
\frac{\left|\mathcal{R}f(x)\right|^{p}}{|x|^{\alpha p}}
dx\right)^{\frac{1}{p}} \left(\int_{\mathbb{G}}
\frac{|f(x)|^{p}}{|x|^{\frac{\beta p}{p-1}}}
dx\right)^{\frac{p-1}{p}},
$$
where we have used H\"older's inequality. Thus, we arrive at
\begin{equation}\label{calcul}
\left|\frac{Q-\gamma}{p}\right| \int_{\mathbb{G}}
\frac{|f(x)|^{p}}
{|x|^{\gamma}}dx \leq \left(\int_{\mathbb{G}}
\frac{\left|\mathcal{R}f(x)\right|^{p}}{|x|^{\alpha p}}
dx\right)^{\frac{1}{p}} \left(\int_{\mathbb{G}}
\frac{|f(x)|^{p}}{|x|^{\frac{\beta p}{p-1}}}
dx\right)^{\frac{p-1}{p}}.
\end{equation}

Now let us show the sharpness of the constant. We need to examine the equality
condition in above H\"older's inequality as in the Euclidean case (see \cite{DJSJ-2013}).
Consider the function
\begin{equation}
g(x)=\left\{
\begin{array}{ll}
    e^{-\frac{C}{\lambda}|x|^{\lambda}},\quad \lambda:=\alpha-\frac{\beta}{p-1}+1\neq0,\\
    \frac{1}{|x|^{C}},\quad \alpha-\frac{\beta}{p-1}+1=0, \\
\end{array}
\right.
\label{EQ:fs}
\end{equation}
where $C=\left|\frac{Q-\gamma}{p}\right|$ and $\gamma\neq Q.$
Then it can be checked that
\begin{equation}
\left|\frac{p}{Q-\gamma}\right|^{p}
\frac{\left|\mathcal{R}g(x)\right|^{p}}{|x|^{\alpha p}}=
\frac{|g(x)|^{p}}{|x|^{\frac{\beta p}{p-1}}},
\end{equation}
which satisfies the equality
condition in H\"older's inequality.
This shows that the constant $C=\left|\frac{Q-\gamma}{p}\right|$ is sharp.
\end{proof}

\section{Higher order cases}
\label{Sec4}
In this section we shortly discuss that by iterating the established $L^{p}$-Caffarelli-Kohn-Nirenberg type inequalities one can get inequalities of higher order. To start let us consider in \eqref{eqLpCKN} the case
$$\beta=\gamma\left(1-\frac{1}{p}\right),$$
that is, taking $\beta=(\alpha+1)(p-1)$ the inequality
\eqref{eqLpCKN} implies that $\gamma=p(\alpha+1)$ and

\begin{equation}\label{Lpweighted}
\left\|\frac{f}{|x|^{\alpha+1}}\right\|_{L^{p}(\mathbb{G})}\leq \frac{p}{|Q-p(\alpha+1)|}
\left\|\frac{1}{|x|^{\alpha}}\mathcal{R} f\right\|_{L^{p}(\mathbb{G})},\quad 1<p<\infty,
\end{equation}
for any $f\in C^{\infty}_{0}(\mathbb{G}\backslash\{0\})$ and all $\alpha\in\mathbb{R}$ with $\alpha\neq \frac{Q}{p}-1.$

Now putting $\mathcal{R} f$ instead of $f$ and $\alpha-1$
instead of $\alpha$ in \eqref{Lpweighted} we consequently have

$$
\left\|\frac{\mathcal{R} f}{|x|^{\alpha}}\right\|_{L^{p}(\mathbb{G})}\leq \frac{p}{|Q-p\alpha|}
\left\|\frac{1}{|x|^{\alpha-1}}\mathcal{R}^{2} f\right\|_{L^{p}(\mathbb{G})},
$$
for $\alpha\neq \frac{Q}{p}.$
Combining it with \eqref{Lpweighted} we get

\begin{equation}\label{iter1}
\left\|\frac{f}{|x|^{\alpha+1}}\right\|_{L^{p}(\mathbb{G})}\leq \frac{p}{|Q-p(\alpha+1)|}\frac{p}{|Q-p\alpha)|}
\left\|\frac{1}{|x|^{\alpha-1}}\mathcal{R}^{2} f\right\|_{L^{p}(\mathbb{G})},
\end{equation}
for each $\alpha\in\mathbb{R}$ such that $\alpha\neq \frac{Q}{p}-1$ and $\alpha\neq \frac{Q}{p}.$
This iteration process gives

\begin{equation}\label{Lph1}
\left\|\frac{f}{|x|^{\theta+1}}\right\|_{L^{p}(\mathbb{G})}\leq A_{\theta,k}
\left\|\frac{1}{|x|^{\theta+1-k}}\mathcal{R}^{k} f\right\|_{L^{p}(\mathbb{G})},\quad 1<p<\infty,
\end{equation}
for any $f\in C^{\infty}_{0}(\mathbb{G}\backslash\{0\})$ and all $\theta\in\mathbb{R}$ such that $\prod_{j=0}^{k-1}
\left|Q-p(\theta+1-j)\right|\neq0,$ and
$$A_{\theta,k}:=p^{k}\left[\prod_{j=0}^{k-1}
\left|Q-p(\theta+1-j)\right|\right]^{-1}.$$
Similarly, we have
\begin{equation}\label{Lph2}
\left\|\frac{\mathcal{R}f}{|x|^{\vartheta+1}}\right\|_{L^{p}(\mathbb{G})}\leq A_{\vartheta,m}
\left\|\frac{1}{|x|^{\vartheta+1-m}}\mathcal{R}^{m+1} f\right\|_{L^{p}(\mathbb{G})},\quad 1<p<\infty,
\end{equation}
for any $f\in C^{\infty}_{0}(\mathbb{G}\backslash\{0\})$ and all $\vartheta\in\mathbb{R}$ such that $\prod_{j=0}^{m-1}
\left|Q-p(\vartheta+1-j)\right|\neq0,$ and
$$A_{\vartheta,m}:=p^{m}\left[\prod_{j=0}^{m-1}
\left|Q-p(\vartheta+1-j)\right|\right]^{-1}.$$
Now putting $\vartheta+1=\alpha$ and $\theta+1=\frac{\beta}{p-1}$  into \eqref{Lph2} and \eqref{Lph1}, respectively, from \eqref{eqLpCKN} we obtain

\begin{prop}
Let $1<p<\infty.$
For any $k,m\in \mathbb{N}$ we have
\begin{equation}\label{Lphighorder}
\frac{|Q-\gamma|}{p}
\left\|\frac{f}{|x|^{\frac{\gamma}{p}}}\right\|^{p}_{L^{p}
(\mathbb{G})}\leq
\widetilde{A}_{\alpha,m}\widetilde{A}_{\beta,k}\left\|\frac{1}{|x|^{\alpha-m}}\mathcal{R}^{m+1} f\right\|_{L^{p}(\mathbb{G})}\left\|\frac{1}{|x|^{\frac{\beta}{p-1}-k}}\mathcal{R}^{k}f\right\|^{p-1}_{L^{p}(\mathbb{G})},
\end{equation}
for any complex-valued function $f\in C^{\infty}_{0}(\mathbb{G}\backslash\{0\})$,  $\gamma=\alpha+\beta+1$, and $\alpha\in\mathbb{R}$ such that $\prod_{j=0}^{m-1}
\left|Q-p(\alpha-j)\right|\neq0,$ and
$$\widetilde{A}_{\alpha,m}:=p^{m}\left[\prod_{j=0}^{m-1}
\left|Q-p(\alpha-j)\right|\right]^{-1},$$
as well as $\beta\in\mathbb{R}$ such that $\prod_{j=0}^{k-1}
\left|Q-p(\frac{\beta}{p-1}-j)\right|\neq0,$ and
$$\widetilde{A}_{\beta,k}:=p^{k(p-1)}\left[\prod_{j=0}^{k-1}
\left|Q-p\left(\frac{\beta}{p-1}-j\right)\right|\right]^{-(p-1)}.$$
\end{prop}
We also highlight the case $p=2$. In this case an interesting feature is that when we have the exact formula for the remainder which yields the sharpness of the constants as well.
We first recall the following estimate and formula.

\begin{thm}[\cite{Ruzhansky-Suragan:identities}]\label{H-high}
Let $Q\geq 3$, $\alpha\in\mathbb{R}$ and $k\in\mathbb N$ be such that we have $\prod_{j=0}^{k-1}
\left|\frac{Q-2}{2}-(\alpha+j)\right|\neq0$.
Then for all complex-valued functions $f\in C^{\infty}_{0}(\mathbb{G}\backslash\{0\})$
we have
\begin{equation}\label{EQ:high-order1}
\left\|\frac{f}{|x|^{k+\alpha}}
\right\|_{L^{2}(\mathbb{G})}\leq
\left[\prod_{j=0}^{k-1}
\left|\frac{Q-2}{2}-(\alpha+j)\right|\right]^{-1}\left\|\frac{1}{|x|^{\alpha}}\mathcal{R} ^{k}f\right\|_{L^{2}(\mathbb{G})},
\end{equation}
where the constant above is sharp, and is attained if and only if $f=0$.

Moreover, for all $k\in\mathbb N$ and $\alpha\in\mathbb{R}$, the following equality holds:
\begin{multline}\label{equality-high-rem}
\left\|\frac{1}{|x|^{\alpha}}\mathcal{R} ^{k}f\right\|^{2}_{L^{2}(\mathbb{G})}=
\left[\prod_{j=0}^{k-1}
\left(\frac{Q-2}{2}-(\alpha+j)\right)^{2}\right]\left\|\frac{f}{|x|^{k+\alpha}}
\right\|^{2}_{L^{2}(\mathbb{G})}
\\+\sum_{l=1}^{k-1}\left[\prod_{j=0}^{l-1}
\left(\frac{Q-2}{2}-(\alpha+j)\right)^{2}\right]\left\|\frac{1}{|x|^{l+\alpha}}\mathcal{R} ^{k-l}f+
\frac{Q-2(l+1+\alpha)}{2|x|^{l+1+\alpha}}\mathcal{R} ^{k-l-1}f\right\|^{2}_{L^{2}(\mathbb{G})}
\\
+\left\|\frac{1}{|x|^{\alpha}}\mathcal{R} ^{k}f+\frac{Q-2-2\alpha}{2|x|^{1+\alpha}}\mathcal{R} ^{k-1}f \right\|^{2}_{L^{2}(\mathbb{G})}.
\end{multline}
\end{thm}

When $p=2,$ Theorem \ref{LpCKN} can be restated that for each $f\in C^{\infty}_{0}(\mathbb{G}\backslash\{0\}),$
and any homogeneous quasi-norm $|\cdot|$ on $\mathbb{G}$ we have
\begin{equation}\label{eqL2CKN}
\frac{|Q-\gamma|}{2}
\left\|\frac{f}{|x|^{\frac{\gamma}{2}}}\right\|^{2}_{L^{2}(\mathbb{G})}\leq
\left\|\frac{1}{|x|^{\alpha}}\mathcal{R} f\right\|_{L^{2}(\mathbb{G})}\left\|\frac{f}{|x|^{\beta}}\right\|_{L^{2}(\mathbb{G})}, \quad \forall\alpha,\,\beta\in \mathbb{R},
\end{equation}
where $\gamma=\alpha+\beta+1$.
Combining \eqref{eqL2CKN} with \eqref{EQ:high-order1} (or \eqref{equality-high-rem}), one can obtain a number of inequalities with sharp constants, for example:

\begin{equation}\label{EQ:a1}
\frac{|Q-\gamma|}{2}
\left\|\frac{f}{|x|^{\frac{\gamma}{2}}}\right\|^{2}_{L^{2}(\mathbb{G})}\leq
C_{j}(\beta,k)\left\|\frac{1}{|x|^{\alpha}}\mathcal{R} f\right\|_{L^{2}(\mathbb{G})}\left\|\frac{1}{|x|^{\beta-k}}\mathcal{R} ^{k}f\right\|_{L^{2}(\mathbb{G})},
\end{equation}
for $\gamma=\alpha+\beta+1$ and all $\alpha,\,\beta\in \mathbb{R}$ and $k\in \mathbb{N}$, such that,
$$C_{j}(\beta,k):=\left[\prod_{j=0}^{k-1}
\left|\frac{Q-2}{2}-(\beta-k+j)\right|\right]^{-1}\neq0,$$
as well as
\begin{equation}\label{EQ:a2}
\frac{|Q-\gamma|}{2}
\left\|\frac{f}{|x|^{\frac{\gamma}{2}}}\right\|^{2}_{L^{2}(\mathbb{G})}\leq
C_{j}(\alpha,k)\left\|\frac{1}{|x|^{\alpha-k}}\mathcal{R} ^{k+1}f\right\|_{L^{2}(\mathbb{G})}\left\|\frac{f}{|x|^{\beta}}\right\|_{L^{2}(\mathbb{G})},
\end{equation}
for $\gamma=\alpha+\beta+1$ and all $\alpha,\,\beta\in \mathbb{R}$ and $k\in \mathbb{N}$, such that,
$$C_{j}(\alpha,k):=\left[\prod_{j=0}^{k-1}
\left|\frac{Q-2}{2}-(\alpha+k+j)\right|\right]^{-1}\neq0.$$
It follows from \eqref{equality-high-rem} that these constants $C_{j}(\beta,k)$ and $C_{j}(\alpha,k)$
in \eqref{EQ:a1} and \eqref{EQ:a2} are sharp.

\end{document}